\documentclass{amsart}
\usepackage{amsfonts}
\usepackage{amssymb}
\usepackage{mathrsfs}
\usepackage{amsmath}
\usepackage{cite}

\newtheorem{theorem}{Theorem}

\newtheorem{remark}[theorem]{Remark}

\newtheorem{corollary}[theorem]{Corollary}

\newtheorem{definition}[theorem]{Definition}

\newtheorem{lemma}[theorem]{Lemma}

\numberwithin{equation}{section} \numberwithin{theorem}{section}

\renewcommand{\oddsidemargin}{5mm}

\begin{document}
\title
{Existence and Regularity \\For The Generalized Mean Curvature Flow Equations}
\author{RongLi Huang$^{1,2}$}
\address{1. School of Mathematical Sciences, Beijing Normal University,
Laboratory of Mathematics and Complex Systems, Ministry of Education,
Beijing 100875, People's Republic of China}
\email{hrl602@mail.bnu.edu.cn}
\address{2. Institute of Mathematics, Fudan University,
Shanghai 200433, People's Republic of China}
\email{huangronglijane@yahoo.cn}
\author{JiGuang Bao$^3$}
\address{3. Corresponding author. School of Mathematical Sciences, Beijing Normal University,
Beijing 100875, People's Republic of China}
\email{jgbao@bnu.edu.cn }
\date{}
\begin{abstract}
By making use of the approximation method, we obtain the existence and regularity of the viscosity
solutions for the generalized mean curvature flow. The asymptotic behavior of the flow is also considered.
In particular, the Dirichlet problem of the degenerate elliptic equation
$$ -|\nabla v|(\mathrm{div}(\frac{\nabla v}{|\nabla v|})+\nu)=0 $$
is solvable in viscosity sense.
\end{abstract}



\maketitle

\section{Introduction}

Let $n<6$, and $D\subset \mathbb{R}^{n+1}$ a bounded domain with a
$C^{2}$ boundary of mean curvature $H\geq 0$ with respect to its
outer unit normal. For $\Omega=\mathbb{R}^{n+1}\setminus D$ and a
nonnegative function $f\in C^{0,1}(\Omega)$, Bernhard Hein
considered the viscosity solutions of the inverse mean curvature
flow (cf. \cite{H})
\begin{equation}\label{1.1}
\left\{ \begin{aligned}
u_{t}-\mathrm{d}\mathrm{i}\mathrm{v}\biggl(\frac{\nabla u}{|\nabla u|}\biggr)+|\nabla u|&=0,
&(x,t)\in \Omega\times (0,+\infty), \\
u&=0, &(x,t)\in \partial D\times (0,+\infty),\\
u&=f(x), &(x,t)\in \Omega \times\{0\}.
\end{aligned} \right.
\end{equation}
Here $u_{t}=\frac{\partial u}{\partial t}, \nabla
 u=\mathrm{g}\mathrm{r}\mathrm{a}\mathrm{d}$ $u$,
 $\mathrm{d}\mathrm{i}\mathrm{v}$ is the divergence operator in
 $\mathbb{R}^{n+1}$.
 He proved that there exists a unique nonnegative weak solution which satisfies (\ref{1.1}).
 And there is a positive constant $C=C(n,D,f)$ such that for $x\in\Omega$ and all $t>0$,
\begin{equation*}
|\nabla u|\leq C,\qquad\quad
-\frac{\sqrt{(n+1)C}}{\sqrt{t}}-2C\leq \frac{\partial u}{\partial
t}\leq \frac{\sqrt{(n+1)C}}{\sqrt{t}}+C.
\end{equation*}
Y.Giga, M.Ohuma and M.Sato  studied the following Neumann problem (cf. \cite{GOS})
\begin{equation}\label{1.2}
\left\{ \begin{aligned}
u_{t}-|\nabla u|\mathrm{d}\mathrm{i}\mathrm{v}\biggl(\frac{\nabla u}{|\nabla u|}\biggr)&=0,
&(x,t)\in D\times (0,+\infty),  \\
\frac{\partial u}{\partial \gamma}&=0, &(x,t)\in \partial D\times (0,+\infty), \\
u&=f(x), &x\in D\times\{0\},
\end{aligned} \right.
\end{equation}
where $\gamma$ is the outer unit normal of $\partial D$ and $f(x)\in C^{2}(\overline {D})$.
They discovered some interesting  properties of the solution $u(x,t)$ (see
Theorem 1.1 in \cite{GOS}) which satisfies (\ref{1.2}) in viscosity sense.

Inspired from \cite{H} and \cite{GOS}, we investigate  the global properties of
solutions of the generalized mean curvature flow equations
\begin{equation}\label{1.3}
u_{t}-|\nabla u|\left(\mathrm{div}\biggl(\frac{\nabla
u}{|\nabla u|}\biggr)+\nu\right)=0, \quad (x,t)\in D\times(0,+\infty),
\end{equation}
where $\nu$ is a constant.  The equation (\ref{1.3}) has a geometric significance because
$\gamma$-level surface $\Gamma(t)$ of $u$ moves by its mean
curvature when $\nu=0$ provided $\nabla u$ does not vanish on
$\Gamma(t)$. Such a motion of surfaces has been studied by many
authors in various conditions (cf. \cite{CGG},\cite{SZ},\cite{ES1},\cite{ES2},\cite{ES3},\cite{ES4}). However,
the uniformly gradient estimates for solutions of (\ref{1.3}) are little
known and crucial for studying the global
properties of viscosity solutions.

 In the present paper we consider the initial and boundary value problem
\begin{equation}\label{1.4}
\left\{ \begin{aligned}
u_{t}-|\nabla u|\left(\mathrm{d}\mathrm{i}\mathrm{v}\biggl(\frac{\nabla u}{|\nabla u|}\biggr)+\nu\right)&=0,
&(x,t)\in D\times (0,+\infty),  \\
u&=h(x),  &(x,t)\in \partial D\times [0,+\infty), \\
u&=g(x),  &x\in D\times\{0\}.
\end{aligned} \right.
\end{equation}
Here $h(x)$ and $g(x)$ are the given functions on $\overline D$.

Our main purposes are to show the existence and regularity of the viscosity
solutions for (\ref{1.4}), study their asymptotic behavior, and prove that
$u(x,t)$ converges to a solution of the Dirichlet problem of degenerate elliptic equation
\begin{equation}\label{1.5}
\left\{ \begin{aligned} -|\nabla v|\left(\mathrm{div}\biggl(\frac{\nabla
v}{|\nabla v|}\biggr)+\nu\right)&=0, &x\in D, \\
v&=h(x), &x\in \partial D,
\end{aligned} \right.
\end{equation}
as $t\rightarrow +\infty$. The solvability of (\ref{1.5}) doses not seem to be easily found in
 literature as far as we knew.

Quite naturally, we always use the following notations
$$ \varphi_{i}=\displaystyle\frac{\partial\varphi}{\partial x_{i}}, \ \ \ \
\varphi_{ij}=\displaystyle\frac{\partial^{2}\varphi}{\partial x_{i}\partial x_{j}}. $$
Firstly we introduce the definition of viscosity solutions.

\begin{definition}\label{de 2.05}
Suppose that $u(x,t)$ is a function in
$C(\overline {D}\!\times\![0,+\infty))$ and satisfies the initial and boundary conditions
of (1.4). If $\varphi\in C^{\infty}(D\times(0,+\infty))$, $(x,t)\in {\Theta}\subset D\times(0,+\infty)$
and $\Theta$ is a bounded open set, which
satisfies
\begin{equation*}
(u-\varphi)(x,t)=\max_{\overline{\Theta}}(u-\varphi),
\end{equation*}
and at $(x,t)$ that
\begin{equation*}
\varphi_{t}\leq \left(\delta_{ij}-\frac{\varphi_{i}\varphi_{j}}{|\nabla
\varphi|^{2}}\right)\varphi_{ij}+\nu|\nabla\varphi| ,\quad
|\nabla\varphi|\neq 0.
\end{equation*}
Or  there exists $\eta=(\eta_{1},\eta_{2},\cdots,\eta_{n+1})$ with
$|\eta|\leq 1$ at $(x,t)$ such that
\begin{equation*}
\varphi_{t}\leq (\delta_{ij}-\eta_{i}\eta_{j})\varphi_{ij} ,\quad
|\nabla\varphi|= 0.
\end{equation*}
Then $u(x,t)$ is viscosity sub-solution of (\ref{1.4}).
\end{definition}

\begin{definition}\label{de 2.05}
Suppose that $u(x,t)$ is a  function in
$C(\overline {D}\!\times\![0,+\infty))$ and satisfies the initial and boundary conditions
of (\ref{1.4}). If $\varphi\in C^{\infty}(D\times(0,+\infty))$,
$(x,t)\in {\Theta}\subset D\times(0,+\infty)$ and $\Theta$ is a bounded open set, which
satisfies
\begin{equation*}
(u-\varphi)(x,t)=\min_{\overline{\Theta}}(u-\varphi),
\end{equation*}
and at $(x,t)$ that
\begin{equation*}
\varphi_{t}\geq \left(\delta_{ij}-\frac{\varphi_{i}\varphi_{j}}{|\nabla
\varphi|^{2}}\right)\varphi_{ij}+\nu|\nabla\varphi| ,\quad |\nabla\varphi|\neq 0.
\end{equation*}
Or  there exists $\eta=(\eta_{1},\eta_{2},\cdots,\eta_{n+1})$ with
$|\eta|\leq 1$ at $(x,t)$ such that
\begin{equation*}
\varphi_{t}\geq (\delta_{ij}-\eta_{i}\eta_{j})\varphi_{ij} ,\quad |\nabla\varphi|= 0.
\end{equation*}
Then $u(x,t)$ is viscosity super-solution of (\ref{1.4}).
\end{definition}

\begin{definition}\label{de 2.05}
If $u(x,t)$ is a viscosity sub-solution and also is a viscosity
super-solution of (\ref{1.4}), then $u(x,t)$ is a viscosity solution of (\ref{1.4}).
\end{definition}

Let us  fix $h(x)=g(x)$ on $\partial D$, $h(x)\in
C^{2}(\partial D)$, $g(x)\in C^{2}(\overline {D})$ and suppose that $D$ is a smooth strictly convex bounded domain in
 $\mathbb{R}^{n+1}$, and $|\nu|<\frac{nH_{0}}{n+1},$
where $H_{0}$ is the positive lower bound of the mean curvature of $\partial D$.

One of the main results in this paper is  the existence and regularity of viscosity solutions of (\ref{1.4}).

\begin{theorem}\label{1.4}
There exists a unique function $u(x,t)$ which satisfies (\ref{1.4}) in viscosity sense and
\begin{equation}\label{1.6}
u\in C(\overline {D}\times [0,+\infty)),\quad u_{t}\in L^{\infty}(D\times
[0,+\infty)),\quad \nabla u\in L^{\infty}(D\times [0,+\infty)),
\end{equation}
\begin{equation}\label{1.7}
\|u\|_{L^{\infty}(D\times[0,+\infty))}+\|\nabla
u\|_{L^{\infty}(D\times[0,+\infty))}+\|u_{t}\|_{L^{\infty}(D\times[0,+\infty))}\leq C,
\end{equation}
\begin{equation}\label{1.8}
\int^{+\infty}_{0}\int_{D}|u_{t}|^{2}dxdt\leq C,
\end{equation}
where the constant $C$ depends only on $n$, $\nu$, $D$, $\|h\|_{C^2(\partial D)}$ and
$\|g\|_{C^2(\overline {D})}$.
\end{theorem}

As an application of Theorem \ref{1.4}, we have

\begin{corollary}\label{1.5}
 Suppose $u(x,t)$ is the  viscosity solution of
(1.4). Then there exists a function $v(x)$ which satisfies
$v(x)\in C(\overline {D}), \nabla v\in L^{\infty}(D)$, such that
\begin{equation}\label{1.9}
\lim_{t\rightarrow +\infty}u(x,t)=v(x),\quad \mathrm{in} \quad
C(\overline {D}).
\end{equation}
and $v(x)$ satisfies (\ref{1.5}) in viscosity sense.
\end{corollary}
\begin{remark}\label{de 2.05}
By Corollary \ref{1.5}  the Dirichlet problem (\ref{1.5}) is solvable. But
we don't know  whether the viscosity solutions of (\ref{1.5}) is unique.
\end{remark}

The second result of this paper is the Liouville-type property of the viscosity solutions.
Suppose  that $D'$ is a smooth convex bounded  domain in $\mathbb{R}^{n}$, such that
\begin{equation*}
D\cap \{(x', x_{n+1})\in \mathbb{R}^{n+1}\mid |x_{n+1}|<m+1\}=D'\times (-m-1,m+1),
\end{equation*}
where $x'=(x_1,\cdots,x_n)$ and $m$ is a positive constant.

\begin{theorem}\label{1.7}
Let $\nu\geq 0$  and suppose that $g(x',x_{n+1})$ is a non-decreasing function of $x_{n+1}$ which satisfies
\begin{equation}\label{1.10}
g(x',x_{n+1})=\lambda,\quad x_{n+1}\geq m ,
\end{equation}
where $\lambda$ is a constant. Then the viscosity solution $u(x,t)$ in $C(\overline {D}\!\times\![0,+\infty))$
of (1.4) satisfies
\begin{equation}\label{1.11}
u(x',x_{n+1},t)=\lambda,\quad x_{n+1}\geq m .
\end{equation}
\end{theorem}

In the next section we construct an approximation problem of (\ref{1.4}) and establish uniform estimates for their classical solutions.
In the last section we present the proof of the main results.

\section{Preliminary Estimates}
Consider the  approximate problem of (1.4) (with $\epsilon \in(0,1)$)
in a smooth strictly convex bounded domain in $\mathbb{R}^{n+1}$
\begin{equation}\label{2.1}
\left\{\begin{aligned}
u_{t}-\sqrt{\epsilon^{2}+|\nabla
u|^{2}}\cdot\left(\mathrm{div}\biggl(\frac{\nabla u}{\sqrt{\epsilon^{2}+|\nabla u|^{2}}}\biggr)+\nu\right)&=0,
&(x,t)\in D\times (0,+\infty), \\
u&=h(x), &(x,t)\in \partial D\times [0,+\infty),\\
u&=g(x), &x\in D\times\{0\}.
\end{aligned} \right.
\end{equation}
We want to use the continuity method to prove the solvability of
(\ref{2.1}) and then obtain the estimates similarly to (\ref{1.7}) and (\ref{1.8}).

In order to solve (\ref{2.1}) we use the following
form of fixed point theorem (cf. \cite{GT}).

\begin{lemma}$\mathrm{(Leray-Schauder)}$\label{2.1}
Suppose that $\mathfrak{B}$ is a  Banach space, $\chi(b,\sigma)$ is
a map from $\mathfrak{B}\times[0,1]$ to $\mathfrak{B}$. If $\chi$
satisfies

(1)\, $\chi$ is continuous and compact.

(2)\, $\chi(b,0)=0,$ $\forall b\in \mathfrak{B}.$

(3)\, There exists constant $C>0$ such that
$$\|b_{0}\|_{\mathfrak{B}}\leq C,\quad \forall b_{0}\in\{b\in
\mathfrak{B}| \exists \sigma\in [0,1],\quad b=\chi(b,\sigma)\}.$$
Then there exists $b_{0}\in \mathfrak{B}$ such that
$\chi(b_{0},1)=b_{0}$.
\end{lemma}

For any $T>0$, we define
$$\mathfrak{B}=\{u| u\in C(\overline {D}\times[0, T)),\quad \nabla u\in C(\overline {D}\times[0, T))\},
\quad D_{T}=D\times [0,T).$$
From the theory on linear parabolic equation (cf. \cite{L}) and for any
$\tilde{u}\in \mathfrak{B},$ $\sigma\in [0,1]$, there exists a
unique function $u$, where $u\in \mathfrak{B}$, $u\in
W_{p}^{2,1}(D_{T})$ with any $p>0 $ such that $u$ satisfies
\begin{equation}\label{2.2}
\left\{ \begin{aligned}
u_{t}-\left(\delta_{ij}-\sigma^{2}\frac{\tilde{u}_{i}\tilde{u}_{j}}{\epsilon^{2}+\sigma^{2}|\nabla
\tilde{u}|^{2}}\right)u_{ij} &=\sigma\nu \sqrt{\epsilon^{2}+\sigma^{2}|\nabla \tilde{u}|^{2}},
&(x,t)\in D\times (0, T), \\
u&=\sigma h(x), &(x,t)\in \partial D\times [0, T),\\
u&=\sigma g(x), &x\in D\times\{0\}.
\end{aligned} \right.
\end{equation}
By (\ref{2.2}) we can define a map  from $\mathfrak{B}\times[0,1]$
to $\mathfrak{B}$ and denote $u=\chi (\tilde{u},\sigma).$ The main
step in our argument is to validate the three conditions of Lemma
\ref{2.1} one by one.

It is obvious  that  $\chi (\tilde{u},0)=0$ for every
$\tilde{u}\in \mathfrak{B}$ by the uniqueness of the initial and boundary value
problem (\ref{2.2}). We see that the map $\chi$ is compact by Schauder estimates and
Sobolev embedding theorem (cf. \cite{L}). Consequently  we  claim that
$\chi$  is continuous. In deed, this fact follows from the
compactness of $\chi$ and  the uniqueness of the mapping $\chi
(\tilde{u},\sigma)$.

So it remains to  verify the third condition  for applying
Lemma \ref{2.1} to the problem (\ref{2.2}).  Suppose $\chi (u,\sigma)=u$. It
follows from (\ref{2.2}) that $u$ satisfies
\begin{equation}\label{2.3}
\left\{ \begin{aligned}
u_{t}-\sqrt{\epsilon^{2}+\sigma^{2}|\nabla u|^{2}}\cdot\left(\mathrm{div}\biggl(\frac{\nabla
u}{\sqrt{\epsilon^{2}+\sigma^{2}|\nabla u|^{2}}}\biggr)+\sigma\nu\right)&=0,
&(x,t)\in D\times (0, T), \\
u&=\sigma h(x),  &(x,t)\in \partial D\times [0, T),\\
u&=\sigma g(x),  &x\in D\times\{0\}.
\end{aligned} \right.
\end{equation}
By using regularity theory, $u\in C^{\infty}(D_{T})\cap C^{2.1}(\overline {D}_{T})$.  Then the condition (3)
in Lemma \ref{2.1} is equivalence to the boundness of $u$ and
$\nabla u$ in the $L^{\infty}$ norm which is independence
of $\sigma $ if $u\in C^{\infty}(D_{T})\cap C^{2.1}(\overline {D}_{T})$
and $u$ satisfies (\ref{2.3}).

In this section we derive  $W^{1,\infty}$ estimates for the
classical solutions of (\ref{2.3}) in which the bound is not only
independent of $\sigma$, but also independent of $\epsilon$ and $T$.

Set
\begin{equation}\label{2.4}
L_{\sigma}u=u_{t}-\sqrt{\epsilon^{2}+\sigma^{2}|\nabla
u|^{2}}\cdot\left(\mathrm{div}\biggl(\frac{\nabla
u}{\sqrt{\epsilon^{2}+\sigma^{2}|\nabla
u|^{2}}}\biggr)+\sigma\nu\right),
\end{equation}
and
$$ \partial_{p}D_{T}=(\partial D\times[0, T))\cup (D\times\{t=0\}).$$
The estimates follow from the next three lemmas.
The following comparison principle is by Theorem 14.1 in \cite{L}.
\begin{lemma}\label{2.2}
Suppose that $u_{1},u_{2}\in C^{2,1}(D\times(0,T ))\cap
C(\overline {D}\times[0, T)).$ If
$$L_{\sigma} u_{1}\geq L_{\sigma}u_{2},\quad
u_{1}|_{\partial_{p}D_{T}}\geq u_{2}|_{\partial_{p}D_{T}},$$ then
$$u_{1}|_{D_{T}}\geq
u_{2}|_{D_{T}}.$$
\end{lemma}

 The estimates of the maximum norm for the solutions of (\ref{2.3}) is the
 following:

\begin{lemma}\label{2.3}
If $u\in C^{\infty}(D\times(0,
 T))\cap C(\overline {D}\times[0, T))$ is a solution of (\ref{2.3}). Then
\begin{equation}\label{1.03}
\|u\|_{L^{\infty}(D\times[0, T))}\leq C,
\end{equation}
where $C$ is depending only  on $\|h\|_{C(\partial D)},$
$\|g\|_{C(\overline {D})},$ and $D$.
\end{lemma}

\begin{proof}
Step 1. By $|\nu|<\frac{nH_{0}}{n+1}$ and  Theorem 16.10 in \cite{GT},
there exists $\alpha>0$, $v^{\epsilon}\in C^{2+\alpha}(\overline {D})$,
such that
\begin{equation*}\label{1.03}
\left\{ \begin{aligned}
-\sqrt{\epsilon^{2}+\sigma^{2}|\nabla v^{\epsilon}|^{2}}\cdot\left(\mathrm{div}\biggl(\frac{\nabla
v^{\epsilon}}{\sqrt{\epsilon^{2}+\sigma^{2}|\nabla v^{\epsilon}|^{2}}}\biggr)+\sigma \nu\right)&=0,
&x\in D, \\
v^{\epsilon}&=1, &x\in \partial D.
\end{aligned} \right.
\end{equation*}
Set $\displaystyle w=\frac{\sigma}{\epsilon}v^{\epsilon}$. Then $w$
is a classical solution of the following Dirichlet problem:
\begin{equation*}\label{1.03}
\left\{ \begin{aligned}
\mathrm{div}\biggl(\frac{\nabla w}{\sqrt{1+|\nabla w|^{2}}}\biggr)+\sigma^{2}\nu &=0, &x\in D, \\
w&=\frac{\sigma}{\epsilon}, &x\in \partial D.
\end{aligned} \right.
\end{equation*}
It follows from Theorem 6.1 in \cite{CW} that there exists a constant $C$
depending only on $n$ and $diam D$ such that
\begin{equation*}\label{1.03}
\max_{\overline {D}}|w|\leq \frac{\sigma}{\epsilon}+C\sigma^{2}\nu.
\end{equation*}
So
\begin{equation}\label{2.6}
\max_{\overline {D}}|v^{\epsilon}|\leq 1+\epsilon\sigma\nu C\leq C.
\end{equation}

Step 2. Suppose that $\kappa$ is a positive constant which will be determined
later. Set $v_{1}^{\epsilon}= v^{\epsilon}+\kappa$. Then
$v_{1}^{\epsilon}$ satisfies
\begin{equation*}\label{1.03}
\left\{ \begin{aligned} -\sqrt{\epsilon^{2}+\sigma^{2}|\nabla v_{1}^{\epsilon}|^{2}}\cdot
\left(\mathrm{div}\biggl(\frac{\nabla v_{1}^{\epsilon}}{\sqrt{\epsilon^{2}+\sigma^{2}|\nabla v_{1}^{\epsilon}|^{2}}}\biggr)+\sigma\nu\right)
&=0,
&x\in D, \\
v_{1}^{\epsilon}&=1+\kappa, &x\in\partial D.
\end{aligned} \right.
\end{equation*}
By (\ref{2.6}) we can choose $\kappa$ depending only on $\|h\|_{C(\partial
D)},$ $\|g\|_{C(\overline {D})},$ and $D$ such that
\begin{equation*}\label{1.03}
v_{1}^{\epsilon}(x)\geq g(x),\quad v_{1}^{\epsilon}(x)\geq
h(x),\quad x\in \overline {D}.
\end{equation*}
By applying Lemma \ref{2.2} we arrive at
\begin{equation*}\label{1.03}
u(x,t)\leq v_{1}^{\epsilon}(x)\leq C+\kappa\leq C,\quad (x,t)\in
\overline {D}\times [0, T).
\end{equation*}
For the same reason we obtain
\begin{equation*}\label{1.03}
u(x,t)\geq -C,\quad (x,t)\in \overline {D}\times [0, T).
\end{equation*}
This yields the desired results.
\end{proof}

The following is the gradient estimates for the solutions of (\ref{2.3}).

\begin{lemma}\label{2.4}
If $u\in C^{\infty}(D\times(0,
T))\cap C(\overline {D}\times[0, T))$ and is a solution of (\ref{2.3}). Then
\begin{equation}\label{2.7}
\|\nabla u\|_{L^{\infty}(D\times[0, T))}\leq C.
\end{equation}
where $C$ is depending only  on $\|h\|_{C^{2}(\partial D)},$
$\|g\|_{C^{1}(\overline {D})}$ and $D$.
\end{lemma}
\begin{proof}
Step 1.  We derive the gradient estimates of $u$ at the boundary and
the methods comes from \cite{SZ}. Set $w=u-h$. Then by (\ref{2.3}) $w$ satisfies
the following equations on $D\times (0, T)$:
\begin{equation*}
\mathfrak{L}w \triangleq
w_{t}-\left(\delta_{ij}-\sigma^{2}\frac{(w_{i}+h_{i})(w_{j}+h_{j})}{\epsilon^{2}+\sigma^{2}|\nabla w+\nabla
h|^{2}}\right)(w_{ij}+h_{ij})-\sigma\nu\sqrt{\epsilon^{2}+\sigma^{2}|\nabla w+\nabla h|^{2}}=0.
\end{equation*}
In the neighborhood $\Theta$ of $\partial D\times [0, T)$ we will
construct the functions $\psi^{\pm}$ which are independent  of $t$
and satisfy
\begin{equation}\label{2.8}
 \pm \mathfrak{L}\psi^{\pm}\geq 0,\quad (x,t)\in \Theta\cap(D\times (0, T)),
\end{equation}
\begin{equation}\label{2.9}
\psi^{\pm}=w=0,\quad (x,t)\in \Theta\cap(\partial D\times [0,T)),
\end{equation}
\begin{equation}\label{2.10}
 \psi^{-}\leq w\leq \psi^{+},\quad (x,t)\in(\partial \Theta\cap (D\times [0, T)))\cup(\Theta\cap(D\times \{0\})).
\end{equation}
Consequently by Lemma \ref{2.2} we have
\begin{equation}\label{2.11}
\psi^{-}\leq w\leq \psi^{+}, \quad (x,t)\in \bar{\Theta}\cap (D\times [0, T)).
\end{equation}
For $(x,t)\in \partial D\times [0, T),$ if $\vec{a}$ is the normal
vector of $\partial D$ such that
$$ x+s\vec{a}\in\bar{\Theta}\cap (D\times [0, T), \qquad\mathrm{when}\quad 0<s\leq 1.$$
Then by (\ref{2.9}) and (\ref{2.11}) we obtain
\begin{equation*}\label{1.03}
\frac{\psi^{-}(x+s\vec{a})-\psi^{-}(x)}{s}\leq
\frac{w(x+s\vec{a},t)-w(x,t)}{s}\leq
\frac{\psi^{+}(x+s\vec{a})-\psi^{+}(x)}{s}.
\end{equation*}
Letting $s\rightarrow 0$, we have
\begin{equation*}\label{1.03}
\frac{\partial \psi^{-}}{\partial \vec{a}}(x)\leq\frac{\partial
w}{\partial \vec{a}}(x,t)\leq \frac{\partial \psi^{+}}{\partial
\vec{a}}(x).
\end{equation*}
A direct calculation yields on $\Theta\cap(\partial D\times [0,T))$
\begin{equation}\label{2.12}\begin{aligned}
|\nabla u|&\leq \|\nabla w\|_{C(\Theta\cap(\partial D\times [0,T))}+\|\nabla h\|_{C(\Theta\cap(\partial D\times [0,T)))}\\
&\leq \|\nabla \psi^{+}\|_{C(\Theta\cap(\partial D\times [0,T)))}
+\|\nabla \psi^{-}\|_{C(\Theta\cap(\partial D\times [0,T)))}+\|\nabla h\|_{C(\Theta\cap(\partial D\times [0,T)))}\leq C .
\end{aligned}
\end{equation}

In the following we constitute $\psi^{+}$ and $\psi^{-}$ which
satisfy (\ref{2.8})--(\ref{2.10}) in detail. Firstly set
$$\psi^{+}(x)=\lambda d(x), \ x\in \overline {D}, \ \ \ \
N=\{x\in D | d(x)<\rho\}, $$
where $d(x)$ is the distance from $x$ to $\partial D$, $\rho$ and  $\lambda$ are positive
constant which will be determined later. Selecting  the positive
constant $\rho$ to be small enough such that $d(x)$ satisfies

$(a)$ $d(x)\in C^{2}(N)$.

$(b)$ In $N$, $|\nabla d|=1$, and
$$ \sum^{n+1}_{i=1}d_{i}d_{ij} =0, \ \ j=1,2,\cdots,n+1. $$

$(c)$ If $x\in N$, then there exists $x_{0}\in
\partial D$ such that $d(x)=|x-x_{0}|$. By Lemma 14.17
in \cite{GT} we can present the formula
\begin{equation}\label{2.13}
-\triangle d(x)=\sum^{n}_{i=1}\frac{k_{i}}{1-k_{i}d(x)},
\end{equation}
where $k_{1}, k_{2},\cdots,k_{n}$ are the principle curvature of
$\partial D$ at $x_{0}$.

Because $\partial D$ is strictly convex, the mean curvature of
$\partial D$ have the positive lower bound and we denote it by $H_{0}$.
Choosing $\rho<\frac1{H_0}$ and by (\ref{2.13}) we have
\begin{equation}\label{2.14}
\triangle d(x)\leq -nH_0 ,\quad x\in N.
\end{equation}
Now we verify $\psi^{+}$ satisfying (\ref{2.8})--(\ref{2.10}) one by one.

(1) By the definition of $d(x)$, $\psi^{+}$ satisfies (\ref{2.9}).

(2) If $x\in N$,  then we can choose
$x_{0}\in
\partial D$ such that  $d(x)=|x-x_{0}|$.  And by $w(x_{0},0)=0$
we obtain $$w(x,0)=g(x)-h(x)-[g(x_{0})-h(x_{0})]\leq
\beta|x-x_{0}|=\beta d(x),$$ where $\beta$ is  depending only on
$\|h\|_{C^{1}(\partial D)}$ and $\|g\|_{C^{1}(\overline {D})}.$  On the other
hand, if $x\in \partial N\cap D$, then $d(x)=\rho $. So we can select the positive
constant $\lambda$, such that $\psi^{+}$ satisfies (\ref{2.10}).

(3) $\psi^{+}$ satisfies (\ref{2.8}). In fact,
\begin{equation*}\label{1.03}
\mathfrak{L}\psi^{+}=-\lambda\triangle d-\triangle h
+\sigma^{2}\left(\frac{\lambda^{2}d_{i}d_{j}+\lambda d_{i}h_{j}+\lambda
h_{i}d_{j}+h_{i}h_{j}}{\epsilon^{2}+\sigma^{2}|\lambda\nabla d+\nabla h|^{2}}\right)(\lambda d_{ij}+h_{ij})
\end{equation*}
\begin{equation*}\label{1.03}
-\sigma\nu\sqrt{\epsilon^{2}+\sigma^{2}|\nabla
h|^{2}+2\lambda\sigma^{2}\nabla d\cdot\nabla
h+\sigma^{2}\lambda^{2}} .
\end{equation*}
Then by (\ref{2.14}) and $\Sigma^{n+1}_{i=1}d_{ij}d_{i}=0$ we have
\begin{equation*}\label{1.03}
\mathfrak{L}\psi^{+}\geq n\lambda H_{0}-
\|h\|_{C^{2}(\overline {D})}+\sigma^{2}\left(\frac{\lambda^{2}d_{i}d_{j}h_{ij}+2\lambda
h_{i}d_{j}h_{ij}+h_{i}h_{j}h_{ij}+\lambda d_{ij}h_{i}h_{j}
}{\epsilon^{2}+\sigma^{2}\lambda^{2}+2\sigma^{2}\lambda\nabla
d\cdot\nabla h+\sigma^{2}|\nabla h|^{2}}\right)
\end{equation*}
\begin{equation*}\label{1.03}
-\sigma\nu\sqrt{\epsilon^{2}+\sigma^{2}|\nabla
h|^{2}+2\sigma^{2}\lambda\nabla d\cdot\nabla
h+\sigma^{2}\lambda^{2}}.
\end{equation*}
By Lemma 14.17 in \cite{GT},  $|d_{ij}|$ have an upper bound depending
only on $\partial D$. Let positive constant $\lambda$ to be large
enough then  we obtain
\begin{equation}\label{2.15}
\mathfrak{L}\psi^{+}\geq n\lambda H_{0}- \lambda\sigma^{2} |\nu | -C\geq
n\lambda H_{0}- \lambda |\nu | -C,
\end{equation}
where $C$ is depending only on $\partial D,\|h\|_{C^{2}(\partial D)}$.
From (\ref{2.15}) and $|\nu|<nH_0$ let $\lambda$ to be large enough which is depending only on
$\partial D$  and $\|h\|_{C^{2}(\partial D)}$ then we have
\begin{equation*}\label{1.03}
\mathfrak{L}\psi^{+}\geq 0 .
\end{equation*}

For the same reason we can construct $\psi^{-}$ which satisfies (2.8)--(2.10). So we have obtained
the desired results of step 1 by (\ref{2.12}).

Step 2. For $i\in \{1,2,\cdots,n+1\}$, let $\varpi=u_{i}$.
Differentiating (\ref{2.3}) with respect to $x_{i}$ we get
\begin{equation*}\label{1.03}
\varpi_{t}-a^{kl}\varpi_{kl}-b^{l}\varpi_{l}=0,\quad (x,t)\in
D\times(0,T),
\end{equation*}
where
$$a^{kl}=\delta_{kl}-\frac{\sigma^{2}u_{k}u_{l}}{\epsilon^{2}+\sigma^{2}|\nabla u|^{2}} ,$$
$$b^{l}=\frac{2\sigma^{4}u^{\epsilon}_{k}u^{\epsilon}_{m}u^{\epsilon}_{km}u_{l}}{(\epsilon^{2}+\sigma^{2}|\nabla u|^{2})^{2}}
-\frac{2\sigma^{2}u_{k}u_{kl}}{\epsilon^{2}+\sigma^{2}|\nabla
u|^{2}}-\frac{\nu\sigma^{3}u_{l}}{\sqrt{\epsilon^{2}+\sigma^{2}|\nabla u|^{2}}} .$$
By the maximum principle for linear parabolic equation (cf.
\cite{L}) and (\ref{2.12}) we obtain (\ref{2.8}).
\end{proof}

From Lemma \ref{2.1}--\ref{2.4} and the Schauder estimates we  conclude that
\begin{theorem}\label{2.5} For any $\epsilon>0$, there exists
$u^{\epsilon}$ which satisfies
$$u^{\epsilon}\in C^{\infty}(D\times (0,+\infty)),\quad u^{\epsilon}\in C(\overline {D}\times
[0,+\infty)),\quad \nabla u^{\epsilon}\in C(\overline {D}\times[0,+\infty)),$$
and $u^{\epsilon}$ is a classical solution of (\ref{2.1}). And there holds
$$\|u^{\epsilon}\|_{L^{\infty}(D\times[0,+\infty))}\leq C,\quad
\|\nabla u^{\epsilon}\|_{L^{\infty}(D\times[0,+\infty))}\leq C,$$
where $C$ depends only on $\|h\|_{C^{2}(\partial D)},$
$\|g\|_{C^{1}(\overline {D})},$ $H_{0}$ and $D$.
\end{theorem}

\begin{corollary}\label{2.6}
Suppose $u^{\epsilon}$ is a classical solution of (\ref{2.1}). Then there holds
\begin{equation}\label{2.16}
\int^{+\infty}_{0}\int_{D}|u_{t}^{\epsilon}|^{2}dxdt\leq C ,
\end{equation}
where $C$ depends only on $\|h\|_{C^{2}(\partial D)},$
$\|g\|_{C^{1}(\overline {D})},$ $H_{0}$ and $D$.
\end{corollary}
\begin{proof}
Set
\begin{equation*}\label{1.03} J(t)=\int_{D}\sqrt{|\nabla
u^{\epsilon}|^{2}+\epsilon^{2}}dx .
\end{equation*}
Then
\begin{equation}\label{2.17}
J'(t)=\int_{D}\frac{\nabla u^{\epsilon}\cdot \nabla
u^{\epsilon}_{t}}{\sqrt{|\nabla
u^{\epsilon}|^{2}+\epsilon^{2}}}dx=-\int_{D}\mathrm{div}\frac{\nabla
u^{\epsilon} }{\sqrt{|\nabla
u^{\epsilon}|^{2}+\epsilon^{2}}}u^{\epsilon}_{t}dx .
\end{equation}
From (\ref{2.1}) we see that
\begin{equation}\label{2.18}
 \mathrm{div}\frac{\nabla
u^{\epsilon} }{\sqrt{|\nabla
u^{\epsilon}|^{2}+\epsilon^{2}}}=\frac{u^{\epsilon}_{t}}{\sqrt{|\nabla
u^{\epsilon}|^{2}+\epsilon^{2}}}-\nu.
\end{equation}
Substituting (\ref{2.18}) into (\ref{2.17}) we obtain
\begin{equation}\label{2.19}
J'(t)+\int_{D}\frac{|u_{t}^{\epsilon}|^{2}}{\sqrt{|\nabla
u^{\epsilon}|^{2}+\epsilon^{2}}}dx=\nu\int_{D}u_{t}^{\epsilon}dx .
\end{equation}
For (\ref{2.19}) integrating from $0$ to $T$ and using (\ref{2.5}) we have
\begin{equation*}\label{1.03}\aligned
\int^{T}_{0}\int_{D}\frac{|u_{t}^{\epsilon}|^{2}}{\sqrt{|\nabla
u^{\epsilon}|^{2}+\epsilon^{2}}}dxdt&=J(0)-J(t)+\nu\int_{D}u^{\epsilon}|_{t=T}dx-\nu\int_{D}u^{\epsilon}|_{t=0}dx\\
&\leq J(0)+C ,\endaligned
\end{equation*}
where $C$ is a constant which is independent of $\epsilon$. Taking
$T\rightarrow +\infty$ we get
\begin{equation}\label{2.20}
\int^{+\infty}_{0}\int_{D}\frac{|u_{t}^{\epsilon}|^{2}}{\sqrt{|\nabla
u^{\epsilon}|^{2}+\epsilon^{2}}}dxdt\leq C.
\end{equation}
Combining (\ref{2.7}) with (\ref{2.20}) we arrive at
\begin{equation*}\label{1.03}\aligned
\int^{+\infty}_{0}\int_{D}|u_{t}^{\epsilon}|^{2}dxdt&=\int^{+\infty}_{0}\int_{D}\frac{|u_{t}^{\epsilon}|^{2}}{\sqrt{|\nabla
u^{\epsilon}|^{2}+\epsilon^{2}}}\sqrt{|\nabla u^{\epsilon}|^{2}+\epsilon^{2}}dxdt\\
&\leq (\|\nabla u^{\epsilon}\|_{L^{\infty}(D\times[0,+\infty))}+\epsilon)\int^{+\infty}_{0}\int_{D}\frac{|u_{t}^{\epsilon}|^{2}}{\sqrt{|\nabla
u^{\epsilon}|^{2}+\epsilon^{2}}}dxdt\\&\leq C,
\endaligned
\end{equation*}
where $C$ depends only on $\|h\|_{C^{2}(\partial D)}$, $\|g\|_{C^{1}(\overline {D})}$, $H_{0}$ and $D$.
\end{proof}

\begin{corollary}\label{2.7}
Suppose $u^{\epsilon}$ is a classical solution of
(\ref{2.1}). Then there holds
\begin{equation}\label{2.21}
\|u_{t}^{\epsilon}\|_{L^{\infty}(D\times[0,+\infty))}\leq C.
\end{equation}
where $C$ depends only on $\|g\|_{C^{2}(\overline {D})}.$
\end{corollary}
\begin{proof}
Set $\omega=u^{\epsilon}_{t}$. Differentiating (\ref{2.1}) with respect to
$t$ we get
\begin{equation*}\label{1.03}
\omega_{t}-a^{kl}\omega_{kl}-b^{l}\omega_{l}=0,\quad (x,t)\in
D\times(0,+\infty),
\end{equation*}
where
$$a^{kl}=\delta_{kl}-\frac{u^{\epsilon}_{k}u^{\epsilon}_{l}}{\epsilon^{2}+|\nabla u^{\epsilon}|^{2}},$$
$$b^{l}=\frac{2u^{\epsilon}_{k}u^{\epsilon}_{m}u^{\epsilon}_{km}u^{\epsilon}_{l}}{(\epsilon^{2}+|\nabla u^{\epsilon}|^{2})^{2}}
-\frac{2u^{\epsilon}_{k}u^{\epsilon}_{kl}}{\epsilon^{2}+|\nabla
u^{\epsilon}|^{2}}-\frac{\nu u^{\epsilon}_{l}}{\sqrt{\epsilon^{2}+|\nabla u^{\epsilon}|^{2}}} .$$
From $u^{\epsilon}|_{\partial D\times[0,T)}=h(x)$, there holds
\begin{equation*}\label{1.03}
\omega=0 ,\quad (x,t)\in \partial D\times [0,+\infty).
\end{equation*}
 From (\ref{2.1}) we see that
\begin{equation*}\label{1.03}
\omega=\sqrt{\epsilon^{2}+|\nabla g|^{2}}\cdot\left(\mathrm{div}\left(\frac{\nabla
g}{\sqrt{\epsilon^{2}+|\nabla
g|^{2}}}\right)+\nu\right),\quad (x,t)\in
D\times \{0\}.
\end{equation*}
This yields (\ref{2.21}) by using maximum principle .
\end{proof}

\section{The proof of main results}

In the third section, we give the proof of Theorem \ref{1.4}, Corollary
\ref{1.5} and Theorem \ref{1.7}.

$\mathbf{Proof \, of \, Theorem \ref{1.4}.}$ \ \
Consider  the classical solution of approximate problem
(\ref{2.1}). From Theorem 2.5 and Corollary \ref{2.7} we see that there exists
$\{\epsilon_{i}\}|^{+\infty}_{i=1}$ satisfying
$\displaystyle\lim_{i\rightarrow +\infty}\epsilon_{i}=0$ such that
there holds
\begin{equation*}\label{1.03}
u^{\epsilon_{i}}\rightarrow u ,\quad\qquad\qquad \mathrm{in}\quad
C(\overline {D}\times [0,+\infty)),
\end{equation*}
\begin{equation*}\label{1.03}
\,\,\,\nabla u^{\epsilon_{i}}\rightharpoonup \nabla u\qquad\qquad
\mathrm{in} \quad L^{\infty}(D\times [0,+\infty)) ,
\end{equation*}
\begin{equation*}\label{1.03}
 \,\,\,u_{t}^{\epsilon_{i}}\rightharpoonup
 u_{t}\qquad\qquad\quad\,\,
\mathrm{in} \quad L^{\infty}(D\times [0,+\infty)) .
\end{equation*}
Combining Corollary \ref{2.6} with Fatou's Lemma we verify  that $u$
satisfies (\ref{1.6})--(\ref{1.8}). On the other hand, by the stability
theorem of viscosity solutions (cf. Theorem 2.4 in \cite{CGG}) $u$ is a
viscosity solution of (\ref{1.4}). This completes the proof of Theorem \ref{1.4}.

$\mathbf{Proof \, of \, Corollary \,\ref{1.5}.}$ \ \
The main idea comes from Y.Giga, M.Ohnuma and M.Sato (cf. \cite{GOS}).

Consider the viscosity solution $u$ of (\ref{1.4}). For $(x,t)\in
 \overline{D_{1}}\triangleq\overline {D}\times [0,1],$ set
\begin{equation*}\label{1.03}
u_{k}(x,t)=u(x,k+t) ,\quad k=1,2,\cdots.
\end{equation*}
From (\ref{1.7}) and the Ascoli-Arzela's Theorem, there exists a
subsequence of $\{u_{k}\}$ (still denote the subsequence by $\{u_{k}\}$) and the function $v(x,t)$,
such that
\begin{equation}\label{3.1}
\lim_{k\rightarrow +\infty}u_{k}(x,t)=v(x,t) ,\quad \mathrm{in}\quad
C(\overline{D_{1}}).
\end{equation}
By (\ref{1.8}) we obtain
\begin{equation*}\label{1.03}
\lim_{k\rightarrow +\infty}\int^{1}_{0}\int_{D}|u_{kt}|^{2}dxdt=\lim_{k\rightarrow +\infty}\int^{k+1}_{k}\int_{D}|u_{t}|^{2}dxdt=0.
\end{equation*}
Letting  $k\rightarrow +\infty$  we have
\begin{equation}\label{3.2}
u_{kt}\rightharpoonup 0,\quad \mathrm{in}\quad L^{2}(D_{1}) .
\end{equation}
It follows from (\ref{3.1}) and (\ref{3.2}) that for any $\phi\in C^{\infty}_{0}(D)$  and $\chi\in
C^{\infty}_{0}(0,1)$,
\begin{equation*}\label{1.03}
\int_{D}\int^{1}_{0}v\phi\chi_{t} dtdx=0.
\end{equation*}
Then
\begin{equation}\label{3.3}
v_t=0,\quad (x,t)\in D_1.
\end{equation}

By (\ref{1.4}) $u_{k}$ satisfies the following equation in viscosity sense
\begin{equation}\label{3.4}
\left\{ \begin{aligned}
u_{kt}-|\nabla u_{k}|\left(\mathrm{div}\biggl(\frac{\nabla u_{k}}{|\nabla
u_{k}|}\biggr)+\nu \right)&=0, &(x,t)\in D_1\times(0,1), \\
 u_{k}&=h(x), &(x,t)\in \partial D_1\times [0,1].
 \end{aligned}\right.
\end{equation}
 From (\ref{3.4})  taking $k\rightarrow +\infty$  and
 using (\ref{1.7}),(\ref{3.1}),(\ref{3.3}) and applying Theorem 2.4 in \cite{CGG} we deduce that  $v$ satisfies
\begin{equation*}\label{1.03}
\left\{ \begin{aligned}
-|\nabla v|\left(\mathrm{div}\biggl(\frac{\nabla v}{|\nabla v|}\biggr)+\nu\right)&=0, &x\in D, \\
v&=h(x), &x\in \partial D,
\end{aligned} \right.
\end{equation*}
in viscosity sense.
This completes the proof of Corollary \ref{1.5}.

$\mathbf{Proof \, of \, Theorem \, \ref{1.7}.}$ \ \
Firstly taking positive constant $\delta$ to be small enough we can
construct a pair of  non-decreasing $C^{2}$ functions $g^{+}(\tau)$
and $g^{-}(\tau)$ such that
\begin{equation}\label{3.5}
\left\{ \begin{aligned}
&g^{-}(x_{n+1})=g^{+}(x_{n+1})=\lambda, &x_{n+1}\geq m+\delta,\\
&g^{-}(x_{n+1})\leq\max_{x'\in \overline{D}}g(x',x_{n+1})\leq g^{+}(x_{n+1}), &x_{n+1}\leq m+\delta.\\
\end{aligned} \right.
\end{equation}
In fact, by the hypothesis of $g$ we can choose $g^{+}(\tau)\equiv\lambda$. Set
$$g_{\varepsilon}(\tau)=\left\{ \begin{aligned}
&\lambda, &\mathrm{if}\quad\tau\geq m+\delta, \\
&\frac{\lambda}{\varepsilon}(\tau-m-\delta+\varepsilon), &\mathrm{if}\quad\tau\leq m+\delta.
\end{aligned} \right.$$
 By  smoothing  the point $(m+\delta,\lambda)$ and letting
$\epsilon=\epsilon(\delta)$ to be small enough  we obtain $g^{-}(\tau)$ which
satisfies (\ref{3.5}).

Let
\begin{equation*}\label{1.03}
u^{+}(x',x_{n+1},t)=g^{+}(x_{n+1}+\nu t),
\end{equation*}
\begin{equation*}\label{1.03}
u^{-}(x',x_{n+1},t)=g^{-}(x_{n+1}).
\end{equation*}
 We claim that $u^{+}(x',x_{n+1},t)$ and $u^{-}(x',x_{n+1},t)$ are
viscosity sub-solution and viscosity super-solution of (\ref{1.4}) respectively. Then
using Theorem 1.7 in \cite{IS}, $\mathrm{i}.\mathrm{e}$, the comparison
principle for the viscosity solution of (\ref{1.4}) which is likely to
Lemma \ref{2.2}, we get
\begin{equation}\label{3.6}
u^{-}(x',x_{n+1},t)\leq u(x',x_{n+1},t)\leq
u^{+}(x',x_{n+1},t),\quad (x',x_{n+1},t)\in \overline {D}\times
[0,+\infty).
\end{equation}
In particular, if $x_{n+1}\geq m+\delta$ for any $\delta>0$, then
$u^{+}(x',x_{n+1},t)=u^{-}(x',x_{n+1},t)\equiv \lambda$ by making use of
(3,5). Taking $\delta\rightarrow 0$ we obtain (\ref{1.11}).

Now we prove that $u^{+}(x',x_{n+1},t)$ is viscosity
super-solution of (\ref{1.4}). In a similar way we can prove that
$u^{-}(x',x_{n+1},t)$ is a viscosity sub-solution of (\ref{1.4}).

In fact, for any $(x,t)\in D\times[0,+\infty)$, if $\varphi\in
C^{\infty}(D\times[0,+\infty))$ and there exists a
neighborhood ${\Theta}$ of $(x,t)$ in $D\times[0,+\infty)$
such that
\begin{equation*}
(u^{+}-\varphi)(x,t)=\min_{\overline{\Theta}}(u^{+}-\varphi).
\end{equation*}
Then at $(x,t)$ we have
\begin{equation}\label{3.7}
u^{+}_{t}-\varphi_{t}\leq 0,\quad \nabla u^{+}=\nabla \varphi,\quad
D^{2}u^{+}\geq D^{2}\varphi.
\end{equation}
If $\nabla \varphi=0$. Then by taking
$\eta=(\eta_{1},\eta_{2},\cdots,\eta_n,\eta_{n+1})=(0,0,\cdots,0,1)$ and using
(\ref{3.7}) we obtain
\begin{equation}\label{3.8}
(\delta_{ij}-\eta_{i}\eta_{j})\varphi_{ij}\leq
(\delta_{ij}-\eta_{i}\eta_{j})u^{+}_{ij}=(1-\eta_{n+1}\eta_{n+1})u^{+}_{n+1,n+1}=0.
\end{equation}
By (\ref{3.7}) and (\ref{3.8}) there holds
\begin{equation*}
\varphi_{t}\geq u^{+}_{t}=\nu(g^+)'=\nu\varphi_{n+1}=0\geq
(\delta_{ij}-\eta_{i}\eta_{j})\varphi_{ij}.
\end{equation*}
On the other hand, if $\nabla \varphi\neq 0$. Then by (\ref{3.7}) we get
\begin{equation}\label{3.9}
\varphi_{i}=u^{+}_{i}=0,\quad u^{+}_{ii}=0, \quad i=1,2,\cdots,n,\quad \varphi_{n+1}=u^{+}_{n+1}\neq 0.
\end{equation}
Combining  (\ref{3.7}) with (\ref{3.9}), we obtain
\begin{equation}\label{3.10}
\left(\delta_{ij}-\frac{\varphi_{i}\varphi_{i}}{|\nabla\varphi|^{2}}\right)\varphi_{ij}=\sum^{n}_{i=1}\varphi_{ii}\leq
\sum^{n}_{i=1}u^{+}_{ii}=0,
\end{equation}
\begin{equation}\label{3.11}
\varphi_{t}\geq u^{+}_{t}=\nu\varphi_{n+1}.
\end{equation}
It follows from (\ref{3.10}) and (\ref{3.11}) that
\begin{equation*}
\left(\delta_{ij}-\frac{\varphi_{i}\varphi_{i}}{|\nabla\varphi|^{2}}\right)\varphi_{ij}+\nu|\nabla
\varphi|\leq \varphi_{t}.
\end{equation*}

So we conclude that $u^{+}(x',x_{n+1},t)$ is viscosity super-solution of (\ref{1.4}). This
completes the proof of Theorem \ref{1.7}.

{\bf Acknowledgements.} This work is supported by the National
Natural Science Foundation of China (10671022) and Doctoral
Programme Foundation of Institute of Higher Education of China
(20060027023).

\end{document}